\documentclass[11pt]{article}
\usepackage[utf8]{inputenc}
\usepackage[english]{babel}
\usepackage{cite}
\usepackage{amssymb,amsmath, dsfont}
\usepackage{comment,vmargin}
\usepackage[square,sort,comma,numbers]{natbib}
\usepackage{algorithmic}
\usepackage{algorithm}
\usepackage{xspace}
\usepackage{ bbold }
\usepackage{amsthm, color}
\usepackage[pdftex]{hyperref}

\newtheorem{lemma}{Lemma}

\newtheorem{theorem}{Theorem}
\newtheorem{corollary}{Corollary}
\newtheorem{definition}{Definition}

\setpapersize{A4}
\setmarginsrb{2.5cm}{2cm}{2.5cm}{2cm}{.4cm}{4mm}{0cm}{7.5mm}

\newcommand{\bigO}{\mathcal{O}}

\newcommand{\smallo}{o}

\newcommand{\esp}{\mathds{E}}
\newcommand{\var}{\operatorname{\mathds{V}ar}}
\newcommand{\smat}[1]{\left( \begin{smallmatrix} #1 \end{smallmatrix} \right)}

\newcommand{\setg}{\mathcal{G}}
\newcommand{\vect}[1]{\mathbf{#1}}
\newcommand{\Line}{\operatorname{Path}}
\newcommand{\Cycle}{\operatorname{Cycle}}
\newcommand{\RR}{\mathds{R}}

\newcommand{\HH}{\mathcal{H}}
\newcommand{\LL}{\mathcal{L}}

\newcommand{\real}{\operatorname{Re}}
\newcommand{\setmg}{\mathcal{MG}}
\newcommand{\MG}{G} 
\newcommand{\MA}{A} 

\newcommand{\edge}[1]{#1}
\newcommand{\seqv}{\operatorname{seqv}}
\newcommand{\Edge}{\operatorname{edge}}

\newcommand{\zu}{\zeta_{\vect{u}}}
\newcommand{\zo}{\zeta_{\vect{1}}}

\title{Limit law for the numbers of components of fixed sizes\\ of graphs with degree one or two}
\author{
  \'Elie de Panafieu
  \thanks{RISC -- Johannes Kepler University, Linz -- Austria.
  \textbf{Email:} \href{mailto:depanafieuelie@gmail.com}{depanafieuelie@gmail.com}} 
  \and 
  Nicolas Broutin
  \thanks{Inria Paris--Rocquencourt, Domaine de Voluceau, 78153 Le Chesnay -- France. 
  \textbf{Email:} \href{mailto:nicolas.broutin@inria.fr}{nicolas.broutin@inria.fr}}
}
\begin{document}
\maketitle

In this note, we answer a question of Giardin\`a, Giberti, van der Hofstad and Prioriello~\cite{GGHP14}.
We consider labelled simple graphs
with all vertices of degree $1$ or $2$.
We denote by $\setg_{n_1, n_2}$ the set of such graphs 
with $n_1$ vertices of degree~$1$,
and $n_2$ vertices of degree~$2$.
Throughout the document, the \emph{size} refers to the number of vertices. 
The random variable $U_j$ counts 
the number of connected components
of size $j$ in a graph drawn 
uniformly at random from $\setg_{n_1, n_2}$.
The main result is that
for any fixed integer $q \geq 2$
and real number $\alpha>0$, for $n_1=2k$ is a large even number 
and $n_2 = \lfloor \alpha n_1 / 2 \rfloor$,
the random vector $(U_2, \ldots, U_q)$ has a Gaussian limit distribution as $k\to\infty$.
In Section~\ref{sec:multi}, we extend those results to multigraphs,
with the distribution induced by the configuration model.
The vector $(U_2, \ldots, U_q)$ has the same Gaussian limit distribution,
while $U_1$ follows a Poisson law.

\begin{theorem}\label{thm:main}
Let $\alpha>0$. For $n_1$ an even let $n_2=\lfloor \alpha n_1/2\rfloor$. 
For $j\ge 2$, let $U_j$ denote the number of connected components of size $j$ in 
a uniformly random graph from $\setg_{n_1,n_2}$. Then, for every $j$, as $n_1\to\infty$ along the 
even integers, 
\[
\esp [U_j]\sim \frac{\alpha^{j-2}}{(1+\alpha)^{j-1}} \frac {n_1}2
\qquad \text{and}\qquad
\var(U_j) = O\left(\frac{n_1}2\right).
\]
Furthermore, for any integer $q\ge 2$, as $n_1\to\infty$ along the even integers, 
the vector
\[
\frac{1}{\sqrt{n_1/2}}
  \left(
  U_j
  - \frac{\alpha^{j-2}}{(1+\alpha)^{j-1}}
  \frac{n_1}{2}
  \right)_{2\le j\le q}
\]
converges in distribution to a multivariate Gaussian $\mathcal{N}(\vect{0}, \HH(\alpha))$,
where the positive semi-definite matrix $\HH(\alpha)=(\HH_{i,j}(\alpha), 2\le i,j\le q)$ 
is given by
\[
  \HH_{i,j}(\alpha) = 
  - \frac{\alpha^{i+j-4}}{(1+\alpha)^{i+j-2}}
  \left( 1 + \frac{(i-2-\alpha)(j-2-\alpha)}{\alpha(1+\alpha)} \right) 
  + \frac{\mathbb{1}_{i=j}}{1+\alpha} \left(\frac \alpha {1+\alpha}\right)^{i-2}.
\]
\end{theorem}

Specifically, in order to prove the theorem, we show that in the neighborhood of the vector $\vect{0}$,
the multivariate Laplace transform of the rescaled random variables converges point-wise
to the Laplace transform of a multivariate Gaussian distribution.

In order to simplify the formulation of the lemmas, we introduce the uniform Landau notation
\[
  f(z, \vect{u}) \underset{z \rightarrow z_0}{=} \bigO_V(g(z))
\]
which means that there exists two constants $K$ and $\delta > 0$
independent of $\vect{u}$ such that 
for all $\vect{u}$ in $V$ and $|z - z_0| \leq \delta$,
\[
  \left| f(z, \vect{u}) \right| \leq K g(z).
\]
As usual with the Landau notation, the limit $z_0$
is often implicit when the context leaves no ambiguity.
This definition extends naturally to the ``\emph{small o}''
Landau notation.

    \section{Expression of the generating function}

To derive the limit distribution of the sizes
of the components in $\setg_{n_1, n_2}$,
we analyze its generating function,
defined below.

\begin{definition}
  Let $q\ge 2$ be a natural number.
  Let $\vect{u}$ denote the vector $(u_2, \ldots, u_q)$,
  and $G_{n_1, n_2}(\vect{u})$ denote the ordinary multivariate generating function
  of graphs in $\setg_{n_1, n_2}$ where for $j=2,3,\dots, q$,
  the variable $u_j$ marks the number of connected components of size $j$.
  Therefore, the number of graphs in $\setg_{n_1, n_2}$
  with $m_j$ components of size $j$ for all $2 \leq j \leq q$ is
  \[
    [u_2^{m_2} \ldots u_q^{m_q}] G_{n_1, n_2}(\vect{u}).
  \]
\end{definition}

The following notations will also prove useful
throughout the paper.

\begin{definition}
  The sequence $(v_{n_1, n_2})$ 
  and the multivariate generating functions
  $\Line(z, \vect{u})$ and $\Cycle(z, \vect{u})$
  are defined by
  \begin{align*}
    v_{n_1, n_2} 
    &= 
    \frac{(n_1+n_2)!}{2^{n_1/2} (n_1/2)!},
    \\
    \Line(z,\vect{u}) 
    &= 
    \frac{1}{1-z} + \sum_{j=2}^q (u_j-1) z^{j-2}, 
    \\
    \Cycle(z, \vect{u}) 
    &= 
    \frac{1}{2} \log \frac{1}{1-z} - \frac{z}{2} - \frac{z^2}{4} + \sum_{j=3}^q (u_j-1) \frac{z^j}{2 j}.
  \end{align*}
\end{definition}

In the following lemma, we derive a simple exact formula
for the generating function of the graphs in $\setg_{n_1, n_2}$
with variables marking the components of sizes from $2$ to $q$.

\begin{lemma} \label{th:Gexact}
  
  The generating function $G_{n_1, n_2}(\vect{u})$ is zero when $n_1$ is odd,
  otherwise, it is given by
  \begin{equation} \label{eq:Gexact}
    G_{n_1, n_2}(\vect{u})
    =
    v_{n_1, n_2}
    [z^{n_2}] e^{\Cycle(z, \vect{u})} \Line(z,\vect{u})^{n_1/2}.
  \end{equation}
\end{lemma}

\begin{proof}
  A component of a graph in $\setg_{n_1, n_2}$ is either
  a non-oriented path of size at least $2$,
  or a non-oriented cycle of size at least $3$.
  For simplicity, in the following we refer to those connected graphs 
  as \emph{paths} and \emph{cycles}.

  Let us first consider vertices of degree $1$ as unlabelled (we will label them later on).
  The number of \emph{oriented} paths with $n_2$ vertices of degree $2$
  is then $n_2!$ and the number of non-oriented cycles $n_2 ! / (2 n_2)$.
  Let the variable $z$ mark the vertices of degree $2$, 
  then the exponential generating functions of oriented paths and non-oriented cycles are respectively
  \begin{align*}
    \sum_{n_2 \geq 0} n_2! \frac{z^{n_2}}{n_2!} &= \frac{1}{1-z},\\
    \sum_{n_2 \geq 3} \frac{n_2!}{2 n_2} \frac{z^{n_2}}{n_2!} 
    &= \frac{1}{2} \log \left( \frac{1}{1-z} \right) - \frac{z}{2} - \frac{z^2}{4}.
  \end{align*}

  Now for all $2 \leq j \leq q$, we introduce the variable $u_j$ 
  to mark the components of size $j$.
  Since a path with $n_2$ vertices of degree $2$ is a connected component of size $n_2+2$,
  in the generating function of oriented paths,
  for $j$ from $2$ to $q$, the $j$th coefficient is multiplied by $u_{j+2}$.
  Similarly, the $j$th coefficient of the generating function of cycles
  is multiplied by $u_j$.
  Finally, the generating functions of oriented paths and non-oriented cycles,
  exponential with respect to $z$ and ordinary
  with respect to all $u_j$, are
  \begin{align*}
    \Line(z, \vect{u}) 
    &= 
    \frac{1}{1-z} + \sum_{j=2}^q (u_j-1) z^{j-2},
    \\
    \Cycle(z, \vect{u}) 
    &= 
    \frac{1}{2} \log \frac{1}{1-z} - \frac{z}{2} - \frac{z^2}{4} + \frac{1}{2} \sum_{j=3}^q (u_j-1) \frac{z^j}{j}.
  \end{align*}

  Note that each path contains exactly two vertices of degree $1$,
  while all the vertices of a cycle have degree $2$.
  Therefore, $\setg_{n_1, n_2}$ is empty when $n_1$ is odd 
  and a graph in $\setg_{n_1, n_2}$ is a set of $n_1/2$ non-oriented paths
  and a set containing an arbitrary number of non-oriented cycles.
  %
  %
  We now add labels to the vertices of degree $1$.
  The set of those labels can be any of the $\frac{(n_1+n_2)!}{n_1! n_2!}$
  subsets of size $n_1$ of $\{1,2,\dots, n_1 + n_2\}$.
  We then need to choose a permutation of size $n_1$
  to associate to each vertex of degree $1$ its label. (Note here that the generating function 
  $\Line(z,\vect{u})$ above counts \emph{oriented} paths, so that the vertices of degree one 
  are distinguished.) Furthermore, each non-oriented path matches exactly
  two oriented paths, so we replace $\Line(z,\vect{u})$ with $\Line(z,\vect{u})/2$.
  Finally, the generating function $G_{n_1, n_2}(\vect{u})$ is
  \[
    G_{n_1, n_2}(\vect{u}) = 
    \frac{(n_1+n_2)!}{n_1! n_2!}
    n_1!
    n_2!
    [z^{n_2}] e^{\Cycle(z, \vect{u})} \frac{(\Line(z, \vect{u})/2)^{n_1/2}}{(n_1/2)!},
  \]
  which reduces to the result of the lemma.
\end{proof}

We will obtain in Lemma~\ref{th:Gapprox} 
a uniform asymptotic estimate of $G_{n_1, n_2}(\vect{u})$ 
using the Fourier--Laplace method.
In the next corollary, we reformulate the exact expression 
derived in Lemma~\ref{th:Gexact}
to adopt a form that is a more adapted to this method.
In particular, the coefficient extraction
is replaced by an integral, and the variables $u_2,u_3,\dots, u_q$
are considered as positive real numbers.

\begin{corollary} \label{th:integralform}
  For any $\zeta\in (0,1)$, and $\vect{u}$ in a neighborhood of $\vect{1}$,
  \[
    G_{n_1, n_2}(\vect{u})
    =
    \frac{v_{n_1, n_2}}{2 \pi} 
    \frac{\Line(\zeta, \vect{u})^{n_1/2}}{\zeta^{n_2}}
    \int_{- \pi}^{\pi}
    A(\theta, \vect{u})
    e^{-\phi(\theta, \vect{u}) n_1/2}  
    d \theta,
  \]
  where $\phi$ and $A$ are defined by
  \begin{align*}
    \phi(\theta, \vect{u})
    &=
    \log(\Line(\zeta, \vect{u})) - \log(\Line(\zeta e^{i \theta}, \vect{u})) + i \alpha \theta,
    \\
    A(\theta, \vect{u})
    &=
    \exp \left( \Cycle(\zeta e^{i \theta}, \vect{u}) \right).
  \end{align*}
\end{corollary}
\begin{proof}
  For every $\vect{u}$, the generating function $G(z,\vect{u})$ has radius of 
  convergence $1$.
  We rewrite the coefficient extraction of Equation~\eqref{eq:Gexact}
  as a Cauchy integral on a circle of radius $\zeta\in (0,1)$,
  whose value will be adjusted later
  \begin{align*} \label{eq:GCauchy}
    G_{n_1, n_2} (\vect{u}) 
    &=
      \frac{v_{n_1, n_2}}{2 \pi} 
      \int_{- \pi}^{\pi} 
      \exp \left( \Cycle(\zeta e^{i \theta}, \vect{u}) \right)
      \frac{\Line(\zeta e^{i \theta}, \vect{u})^{n_1/2}}{(\zeta e^{i \theta})^{n_2}}
      d \theta,
    \\ &=
      \frac{v_{n_1, n_2}}{2 \pi}
      \frac{\Line(\zeta, \vect{u})^{n_1/2}}{\zeta^{n_2}} 
      \int_{- \pi}^{\pi} 
      A(\theta, \vect{u})
      e^{- \phi( \theta, \vect{u}) n_1/2}
      d \theta,
  \end{align*}
  with $\phi$ and $A$ defined as in the lemma.
\end{proof}

The Laplace method requires to locate the minimum
of the function $\theta \mapsto \phi(\theta, \vect{u})$ 
and the behavior of $A$ and $\phi$ in its vicinity.
This information is derived in the following lemma.

\begin{lemma} \label{th:Gexact2}
Let $\alpha>0$, $n_1$ be an even number
and $n_2$ the closest integer to $\alpha n_1/2$.
Let $\zeta$ denote the unique solution in $(0,1)$ of the equation
\begin{equation} \label{eq:zeta}
  \zeta \partial_z \log ( \Line(\zeta, \vect{u}) ) = \alpha.
\end{equation}
%
There exists a neighborhood $V \subset \RR_{> 0}^q$ of $\vect{1}$,
such that the functions $\phi$ and $A$ satisfy the following properties
for $\theta$ in a complex neighborhood of $0$:
\begin{enumerate}
\item \label{item:taylor}
uniformly for $\vect{u} \in V$, it holds that
\begin{align*}
  \phi(\theta, \vect{u}) 
  &= 
  \partial_{\theta}^2 \phi(0, \vect{u}) \frac{\theta^2}{2} + \bigO_V(\theta^3),
  \\
  A(\theta, \vect{u})
  &=
  A(0, \vect{u}) + \bigO_V(\theta),
\end{align*}
\item \label{item:nocancellation}
for all $\vect{u}\in V$, we have $\partial_{\theta}^2 \phi(0, \vect{u})>0$, and $A(0, \vect{u})\ne 0$,
\item \label{item:positiverealpart}
for all $\vect{u}\in V$, the real part of $\phi(\theta, \vect{u})$ is non-negative, 
$\real(\phi(\theta, \vect{u})) \geq 0$, with equality only at $\theta = 0$.
\end{enumerate}
\end{lemma}

Before proceeding to the proof, observe that the value of $\zeta$ in the statement actually 
depends on $\alpha$ and $\vect{u}$. When needed, we shall write $\zu$ instead of $\zeta$ 
to avoid any ambiguity.

\begin{proof}
By definition, $\phi(0,\vect{u})=0$.
We choose $\zeta$ such that
\[
  \partial_{\theta} \phi(\zeta, \vect{u}) = 0,
\]
which is equivalent with Equation~\eqref{eq:zeta}.
A simple computation reduces this last expression to
\begin{equation} \label{eq:zetalong}
  \zeta 
  \frac{1 + (1-\zeta)^2 \sum_{j=3}^q (u_j - 1) (j-2) \zeta^{j-3}}
    {1-\zeta + (1-\zeta)^2 \sum_{j=2}^q (u_j-1) \zeta^{j-2}}
  = \alpha.
\end{equation}
In particular, when $\vect{u} = \vect{1}$, we have
\[
  \zo = \frac{\alpha}{1+\alpha}.
\]

When the components $u_2,u_3,\dots, u_q$ of $\vect{u}$ are positive numbers,
the analytic function 
\[
  z \mapsto z \frac{\partial_z \Line(z, \vect{u})}{\Line(z, \vect{u})}
\]
has positive coefficients, so it is strictly increasing.
Furthermore, using Expression~\eqref{eq:zetalong},
we see that this function tends to $0$ (resp.\ infinity)
when $z$ goes to $0$ (resp.\ $1$) along the real axis.
Therefore, for any positive real numbers $\alpha$ and $u_2, u_3,\dots, u_q$,
Equation~\eqref{eq:zeta} has a unique solution~$\zeta$ in $(0,1)$.
In particular, for $\vect{u} = \vect{1}$, Equation~\eqref{eq:zetalong} becomes
$\zo/(1-\zo) = \alpha,$ which implies that
\[
  \zo = \frac{\alpha}{1+\alpha}.
\]

Equation~\eqref{eq:zeta} defines $\zeta$ implicitly as a function of $\alpha$ and $\vect{u}$,
and has a solution $\zo$ for $\vect{u} = \vect{1}$.
Furthermore, the derivative with respect to $\zeta$
of the left-hand side of Equation~\eqref{eq:zeta}
does not vanish at $\vect{u} = \vect{1}$, since
\[
  \left.
  \partial_\zeta 
  \left( \zeta 
    \frac{\partial_z \Line(\zeta, \vect{u})}
      {\Line(\zeta, \vect{u})} 
  \right)\right|_{\vect{u} = \vect{1}}
  =
  \left.
  \partial
  \left( \zeta 
    \frac{\partial \left( (1 - \zeta)^{-1} \right)}
      {(1-\zeta)^{-1}} 
  \right)\right|_{\zeta = \zeta_\vect{1}}
  =
  (1 + \alpha)^2.
\]
Therefore, according to the Implicit Function Theorem,
for all positive $\alpha$, 
there exists a neighborhood of $\vect{1}$ on which the function $\vect{u} \mapsto \zeta$
is continuous and thus close to $\zo$.

By continuity, for any $\alpha>0$, 
there exists a neighborhood of $(0,\vect{1})$ where the function
\[
  (\theta, \vect{u}) \mapsto |\partial_{\theta}^3 \phi(\theta, \vect{u})|
\]
is bounded.
Using Taylor's Theorem, we conclude that for every $\alpha>0$,
there exists a neighborhood $V$ of $\vect{1}$ such that
\begin{align*}
  \phi(\theta, \vect{u}) &= \partial_{\theta}^2 \phi(0, \vect{u}) \frac{\theta^2}{2} + \bigO_V(\theta^3),\\
  A(\theta, \vect{u}) &= A(0, \vect{u}) + \bigO_V(\theta).
\end{align*}
Since $\partial_{\theta}^2 \phi(0, \vect{1})$ and $A(0, \vect{1})$
are non-zero (the first one is equal to $\alpha (1+\alpha)$, the second an exponential),
by a continuity argument, we can choose $V$ small enough to ensure that
$\partial_{\theta}^2 \phi(0, \vect{u})$ is positive 
and $A(0, \vect{u})$ does not cancel.
\end{proof}

  \section{Asymptotic extraction of the coefficients}

In this section, we obtain the asymptotics for the coefficients of the generating function 
$G_{n_1,n_2}(\vect{u})$. The result relies on the following technical lemma. The idea is 
very classical, but we coud not find a reference to this multidimensional version and we 
include a proof for the sake of completeness. 

\begin{lemma} \label{th:fourierlaplace}
  Let us consider a real vector $\vect{u}_0$ in $\RR^d$,
  a real neighborhood $V \subset \RR^q$ of $\vect{u}_0$
  and $\epsilon>0$.
  Let $A(\theta, \vect{u})$ and $\phi(\theta, \vect{u})$
  denote two continuous complex functions
  satisfying the following properties:
  \begin{enumerate}
  \item
  for all $\vect{u}\in V$, the functions
  $\theta \mapsto A(\theta, \vect{u})$ and
  $\theta \mapsto \phi(\theta, \vect{u})$
  are analytic at the origin and have a radius of convergence
  greater than $\epsilon$,
  \item \label{item:uniftaylor}
  uniformly for $\vect{u} \in V$,
  \begin{align*}
    \phi(\theta, \vect{u}) 
    &= 
    \partial_{\theta}^2 \phi(0, \vect{u}) \frac{\theta^2}{2} + \bigO_V(\theta^3),
    \\
    A(\theta, \vect{u})
    &=
    A(0, \vect{u}) + \bigO_V(\theta),
  \end{align*}
  \item \label{item:nonzero}
  for all $\vect{u}\in V$, we have $\partial_{\theta}^2 \phi(0, \vect{u})>0$
   and $A(0, \vect{u})\neq 0$,
  \item
  for all $\vect{u}\in V$, and $\theta\in [-\epsilon, \epsilon]$, 
  the real part of $\phi(\theta, \vect{u})$ is non-negative, $\real(\phi(\theta, \vect{u})) \geq 0$,
  with equality only at $\theta = 0$.
  \end{enumerate}
  Then there exists a neighborhood $W \subset V$ of $\vect{u}_0$ where
  \[
    \int_{-\epsilon}^{\epsilon} A(\theta, \vect{u}) e^{- n \phi(\theta, \vect{u})} d \theta
    = \frac{\sqrt{2 \pi} A(0, \vect{u})}{\sqrt{n \partial_{\theta}^2 \phi(0, \vect{u})}}
    (1 + \bigO_W(n^{-1/5})).
  \]
\end{lemma}

\begin{proof}
  We follow the same steps as in the proof 
  of the \emph{Large Power Theorem}~\cite[Theorem VIII.8 page 587]{FS09},
  but all the intermediate results need to be uniform
  with respect to $\vect{u}$.
  First, we reduce the domain of integration,
  then the integrand is approximated by a Gaussian integrand
  and finally the domain of integration is extended to $\RR$,
  and we replace the classic Gaussian integral by its value.

  \medskip
  \noindent\textsc{Reduction of the domain of integration.}\ 
  Let $C \subset V$ denote a compact set that contains a neighborhood of $\vect{u}_0$.
  We first show that the main contribution of the integral
  \[
    I_n(\vect{u}) = \int_{- \epsilon}^{\epsilon} A(\theta, \vect{u}) e^{- n \phi(\theta, \vect{u})} d \theta
  \]
  comes from the vicinity of $\theta = 0$.
  Specifically, we introduce the integral
  \[
   \tilde{I}_{n}(\vect{u}) = \int_{- n^{-2/5}}^{n^{-2/5}} A(\theta, \vect{u}) e^{- n \phi(\theta, \vect{u})} d \theta
  \]
  and prove that there exists a constant $K$ independent of $\vect{u}$ such that
  \begin{equation} \label{eq:reducedomain}
    I_n(\vect{u}) 
    = 
    \tilde{I}_{n}(\vect{u})
    + \bigO_C \left( e^{- K n^{1/5}} \right).
  \end{equation}
  We start with the inequality
  \begin{equation} \label{eq:dominateintegral}
    \left| I_n(\vect{u}) - \tilde{I}_{n}(\vect{u}) \right|
    \leq
    \int_{n^{-2/5} \leq | \theta | \leq \epsilon}
    \left| A(\theta, \vect{u}) e^{- n \phi(\theta, \vect{u})} \right| d \theta
  \end{equation}
  where the right-hand side is at most
  \[
    2 \epsilon
    \sup_{\substack{ |\theta| \leq \epsilon \\ \vect{u} \in C}} | A(\theta, \vect{u}) | 
    \exp \left(
      - n \inf_{\substack{n^{-2/5} \leq |\theta| \leq \epsilon \\ \vect{u} \in C}} 
        \real \left( \phi(\theta, \vect{u}) \right) 
    \right).
  \]
  Since the function $(\theta, \vect{u}) \mapsto |A(\theta, \vect{u})|$ is continuous, it reaches a finite maximum
  in the compact set $[-\epsilon, \epsilon] \times C$, so
  \[
    2 \epsilon
    \sup_{\substack{ |\theta| \leq \epsilon \\ \vect{u} \in C}} | A(\theta, \vect{u}) |
    = \bigO_C(1).
  \]
  According to Assumptions~\ref{item:uniftaylor} and~\ref{item:nonzero},
  \[
    \real ( \phi(\theta, \vect{u}) ) = \partial_\theta^2 \phi(0, \vect{u}) \frac{\theta^2}{2} + \bigO_C(\theta^3),
  \]
  so
  \begin{align*}
    n \inf_{\substack{n^{-2/5} \leq |\theta| \leq \epsilon \\ \vect{u} \in C}} 
    \real \left( \phi(\theta, \vect{u}) \right)
    &=
    \left( 
      \frac{1}{2} 
      \inf_{\vect{u} \in C} \partial_\theta^2 \phi(0, \vect{u}) 
    \right)
    n^{1/5}
    + \bigO_C(n^{-1/5}).
  \end{align*}
  We conclude that
  \[
    \int_{n^{-2/5} \leq | \theta | \leq \epsilon}
    \left| A(\theta, \vect{u}) e^{- n \phi(\theta, \vect{u})} \right| d \theta
    =
    \bigO_C \left( e^{- K n^{1/5}} \right)
  \]
  where $K = \frac{1}{2} \inf_{\vect{u} \in C} \partial_\theta^2 \phi(0, \vect{u})>0$.
  Combined with Equation~\eqref{eq:dominateintegral}, this last result
  proves Equality~\eqref{eq:reducedomain}.

  \medskip
  \noindent\textsc{Approximation of the integrand.}\ 
  We inject the expressions of $A(\theta, \vect{u})$ and $\phi(\theta, \vect{u})$
  from Assumption~\ref{item:uniftaylor} in the definition of $\tilde{I}_n(\vect{u})$
  \[
    \tilde{I}_n(\vect{u}) = 
    \int_{-n^{-2/5}}^{n^{-2/5}}
    (A(0,\vect{u}) + \bigO_C(\theta))
    e^{-n \partial_\theta^2 \phi(0, \vect{u}) \theta^2/2+ n \bigO_C(\theta^3)}
    d \theta.
  \]
  Uniformly with respect to $\vect{u}\in C$ 
  and $\theta\in [-n^{-2/5}, n^{-2/5}]$, we have
  \begin{align*}
    e^{n \bigO_C(\theta^3)} 
    &= 
    1 + \bigO_{\vect{u} \in C, |\theta| \leq n^{-2/5}} (n^{-1/5}),
    \\
    A(0,\vect{u}) + \bigO_C(\theta)
    &=
    A(0,\vect{u}) \left(
      1 + \bigO_{\vect{u} \in C, |\theta| \leq n^{-2/5}} (n^{-2/5})
    \right).
  \end{align*}
  Remark that this property holds 
  because we reduced the domain of integration.
  We then obtain
  \[
    \tilde{I}_n(\vect{u}) = 
    \int_{-n^{-2/5}}^{n^{-2/5}}
    A(0,\vect{u})
    e^{-n \partial_\theta^2 \phi(0, \vect{u}) \theta^2/2}
    d \theta
    \left(
      1 + \bigO_C(n^{-1/5})
    \right),
  \]
  which becomes
  \begin{equation} \label{eq:almostdone}
    \tilde{I}_n(\vect{u}) = 
    \frac{A(0,\vect{u})}{\sqrt{n \partial_\theta^2 \phi(0, \vect{u})}}
    \int_{-n^{1/10}}^{n^{1/10}}
    e^{- t^2/2}
    d t
    \left(
      1 + \bigO_C(n^{-1/5})
    \right)
  \end{equation}
  after the linear change of variable 
  \[
    t = \theta \sqrt{n \partial_\theta^2 \phi(0, \vect{u})}.
  \]

  \medskip\noindent\textsc{Gaussian integral.}\ 
  To conclude the proof, we quickly prove the classic fact
  \[
    \int_{-n^{1/10}}^{n^{1/10}} e^{-t^2/2} d t
    =
    \sqrt{2 \pi}
    \left( 1 + \bigO(e^{-n^{1/5}/2}) \right).
  \]
  Indeed, the complete Gaussian integral is $\int_{-\infty}^{\infty} e^{-t^2/2} d t = \sqrt{2 \pi}$
  while
  \[
    \left| 
      \int_{-\infty}^{\infty} e^{-t^2/2} d t
      - \int_{-n^{1/10}}^{n^{1/10}} e^{-t^2/2} d t
    \right| 
    \leq 
    2 \int_{n^{1/10}}^{\infty} t e^{-t^2/2} d t = 2 e^{-n^{1/5}/2}
  \]
  as soon as $n^{1/10}$ is greater than~$1$.
  Injecting this relation in Equations~\eqref{eq:almostdone}
  and~\eqref{eq:reducedomain}, we obtain
  \[
    I_n(\vect{u}) 
    = 
    \frac{\sqrt{2\pi} A(0)}{\sqrt{n \partial_\theta^2 \phi(0, \vect{u})}}
    \left( 1 + \bigO_C(n^{-1/5}) \right)
    \left( 1 + \bigO_C(e^{-n^{1/5}/2}) \right)
    + \bigO_C(e^{- K n^{1/5}})
  \]
  which concludes the proof.
\end{proof}

The error term of the previous lemma
could be improved up to $\bigO_W(n^{-1})$,
but this would require more work.
Actually, in the following we will simply use a $\smallo_W(1)$
error term, which is sufficient for our purpose
and reduces the notations.


Combining Corollary~\ref{th:integralform} 
and Lemmas~\ref{th:Gexact2} and~\ref{th:fourierlaplace},
we obtain the asymptotics of $G_{n_1, n_2}(\vect{u})$
for $u_2,u_3,\dots, u_q$ in small but fixed real neighborhood of $1$, $n_1$ even and   
$2 n_2/n_1$ close to a fixed positive constant $\alpha$.

\begin{corollary} \label{th:Gapprox}
  With the notations of Corollary~\ref{th:integralform} and Lemma~\ref{th:Gexact2},
  for all $\alpha>0$, there is a neighborhood $W$ of $\vect{1}$
  such that, when $n_1$ is even and $n_2 = \lfloor \alpha n_1/2 \rfloor$,
  \[
    G_{n_1, n_2}(\vect{u})
    =
    \frac{v_{n_1, n_2}}{\sqrt{2 \pi}}
    \frac{A(0,\vect{u})}{\sqrt{\partial_\theta^2 \phi(0, \vect{u}) n_1/2}}
    \frac{\Line(\zeta, \vect{u})^{n_1/2}}{\zeta^{n_2}}
    \left( 1 + \smallo_W(1) \right).
  \]
\end{corollary}

    \section{Limit distribution}

We now exploit the expression in Corollary~\ref{th:Gapprox} to obtain the limit distribution
of the vector $(U_2,U_3,\dots, U_q)$ of counts of connected components of sizes $2,3,\dots, q$ 
in a graph drawn uniformly at random from $\setg_{n_1,n_2}$.

\begin{lemma} \label{th:limitdistribution}
  Let $\alpha>0$, $t_2, \ldots, t_q>0$, $n_1$ an even integer 
  and $n_2 = \lfloor \alpha n_1/2 \rfloor$.
For all $2 \leq j \leq q$, let $U_j$ be the random variable
counting the number of connected components of size $j$
in a graph drawn uniformly from $\setg_{n_1, n_2}$,
and let $V_j$ denote the rescaled random variable
\[
  V_j = 
  \frac{1}{\sqrt{n_1/2}}
  \left(
  U_j
  - \frac{\alpha^{j-2}}{(1+\alpha)^{j-1}}
  \frac{n_1}{2}
  \right).
\]
Then the limit when $n_1$ goes to infinity 
of the multivariate Laplace transform of $V_2, \ldots, V_q$ is
\[
  \lim_{n_1 \rightarrow \infty} \LL^{(\vect{V})}_{n_1, n_2}(t_2, \ldots, t_q) 
  = 
  e^{\frac{1}{2} \vect{t} \cdot \HH(\alpha) \cdot \vect{t}}
\]
where $\HH(\alpha)$ is a positive semi-definite symmetric matrix
with rows and columns indexed from $2$ to $q$ 
(i.e. its upper-left coefficient is $\HH_{2,2}(\alpha)$)
and defined by
\[
  \HH_{i,j}(\alpha) = 
  - \frac{\alpha^{i+j-4}}{(1+\alpha)^{i+j-2}}
  \left( 1 + \frac{(i-2-\alpha)(j-2-\alpha)}{\alpha(1+\alpha)} \right) 
  + \frac{\mathbb{1}_{i=j}}{1+\alpha} \left(\frac \alpha {1+\alpha}\right)^{i-2}.
\]
As a consequence $(V_2, \ldots, V_q)$ converges in distribution to
the multivariate Gaussian $\mathcal{N}(\vect{0}, \HH(\alpha))$.
\end{lemma}

\begin{proof}
This lemma is a multivariate version of 
the \emph{Quasi-Power Theorem} of Hwang~\cite{H98},
also available in~\cite[Lemma IX.1 page 646]{FS09},
applied to a particular case. (Note however that the main point of 
Hwang's theorem is the improvement on the rate of convergence; here, 
we only use the same approach but do not try to obtain the best possible rate.)
When $\vect{t} = (t_2, \ldots, t_q)$ is a vector,
the notation $e^{\vect{t}}$
denotes the vector $(e^{t_2}, \ldots, e^{t_q})$.

Note first that the asymptotics of the Laplace transform $\LL_{n_1,n_2}^{(\vect{U})}(\vect{t})$ 
about $\vect{t}=0$ yield asymptotics for the moments of the vector $\vect{U}=(U_2,U_3,\dots, U_q)$. 
We have
\begin{equation} \label{eq:laplace}
  \LL^{(\vect{U})}_{n_1,n_2}(\vect{t})
  =
  B(\vect{t}) e^{n_1\chi(\vect{t})/2}
  \left( 1 + \smallo_W(1) \right)
\end{equation}
where $B(\vect{t})$ and $\chi(\vect{t})$ are defined by
\begin{align*}
  B(\vect{t})
  &=
  \frac{A(0, e^{\vect{t}})}
    {A(0, \vect{1})}
  \sqrt{\frac{\partial_\theta^2 \phi(0, \vect{1})}
    {\partial_\theta^2 \phi(0, e^{\vect{t}})}}
  \\
  \chi(\vect{t})
  &=
  \log \left( \frac{\Line(\zeta_{e^{\vect{t}}}, e^{\vect{t}})}{\Line(\zo, \vect{1})} \right)
  - \alpha
  \log \left( \frac{ \zeta_{e^{\vect{t}}} }{ \zo } \right)
\end{align*}
and are such that $B(\vect{0}) = 1$ and $\chi(\vect{0}) = 0$. So, for $2\le j\le q$,
\begin{align*}
\esp[U_j] 
& = \partial_{t_j} \left.\LL_{n_1,n_2}^{(\vect{U})}(\vect{t})\right|_{\vect{t}=0} 
\sim \frac{n_1}2 \partial_{t_i}\chi(\vect{0})\\
\var(U_j)
& = \left. \partial_{t_j}^2  \LL_{n_1,n_2}^{(\vect{U})}(\vect{t}) 
- (\partial_{t_j} \LL_{n_1,n_2}^{(\vect{U})}(\vect{t}))^2 \right|_{\vect{t}=0}
\sim \frac{n_1}2 \partial_{t_j} B(\vect{0}) \partial_{t_j} \chi(\vect{0}) 
+ \frac{n_1}2 \partial^2_{t_j} \chi(\vect{0}).
\end{align*}

So $\esp[U_j]\sim n_1 \partial_{t_j} \chi(\vect{0})/2$ and $\var(U_j)=O(n_1)$. So
for the limit distribution of $\vect{U}=(U_2,U_3,\dots, U_q)$, the natural rescaling involves
\[
\vect{V}=(V_j)_{2\le j\le q} 
= \left(\frac{U_j - n_1 \partial_{t_j} \chi(\vect{0})/2}{\sqrt {n_1/2}}\right)_{2\le j\le q}.
\]
It now suffices to obtain pointwise convergence of the corresponding Laplace transform
$\LL^{(\vect{V})}_{n_1,n_2}(\vect{t})=\esp[e^{\vect{t} \vect{V}}]$. 
so that, for any $\vect{t}$,
\begin{align*}
  \LL^{(\vect{V})}_{n_1,n_2}(\vect{t})
  &=
  \esp \left[ \exp\left(\sum_{j=2}^q U_j \frac{t_j}{\sqrt{n_1/2}} \right) \right]
  e^{-\sqrt{n_1/2} \nabla \chi(\vect{0}) \cdot \vect{t}} \\
  & = \LL^{(\vect{U})}_{n_1,n_2}\left( \frac{\vect{t}}{\sqrt{n_1/2}} \right)
  e^{-\sqrt{n_1/2} \nabla \chi(\vect{0}) \cdot \vect{t}}.
\end{align*}
In the following, we write $n_1=2k$, for an integer $k$, and we want asymptotics as $k\to\infty$.
However, by definition of $G_{n_1, n_2} (\vect{u})$,
\[
  \LL^{(\vect{U})}_{n_1,n_2}(\vect{t}/\sqrt{k})
  =
  \frac{G_{n_1, n_2}(e^{\vect{t}/\sqrt{k}})}
    {G_{n_1, n_2}(\vect{1})}.
\]
Now, for any fixed $\vect{t}$, for all $n_1$ large enough $\vect{t}/\sqrt{k}$ is in a 
neighborhood of $\vect{0}$, and $\vect{u} = e^{\vect{t}/\sqrt{k}}$ is in a neighborhood 
of $\vect{1}$. Thus, we can apply Lemma~\ref{th:Gapprox} to obtain the asymptotics for 
the Laplace transform $\LL^{(\vect{V})}_{n_1,n_2}(\vect{t})$, as the number of nodes tend 
to infinity, and this uniformly with respect to~$\vect{t}$ in a fixed neighborhood 
$W$ of $\vect{0}$:
\begin{equation} 
  \LL^{(\vect{U})}_{n_1,n_2}(\vect{t}/\sqrt k)
  =
  B(\vect{t}/\sqrt k) e^{k\chi(\vect{t}/\sqrt k)}
  \left( 1 + \smallo_W(1) \right).
\end{equation}

The multivariate Taylor expansion of $\chi(\vect{t})$ 
near $\vect{t} = \vect{0}$ is
\[
  \chi(\vect{t})
  =
  \nabla \chi(\vect{0}) \cdot \vect{t}
  + \frac{1}{2}
  \vect{t} \cdot \HH \cdot \vect{t}
  + \bigO(\|\vect{t}\|^3)
\]
where $\nabla \chi(\vect{0})$ and $\HH$ denote respectively the gradient
and Hessian matrix of $\chi$ at $\vect{0}$. 
Following \eqref{eq:laplace}, we are interested in asymptotics of $k\chi(\vect{t}/k)$, as $k\to\infty$:
\[
  k\chi \left( \frac{\vect{t}}{\sqrt{k}} \right)
  = 
  \sqrt{k} \nabla \chi(\vect{0}) \cdot \vect{t} 
  + \frac{1}{2}
  \vect{t} \cdot \HH \cdot \vect{t}
  + \bigO(k^{-1/2} \|\vect{t}\|^3).
\]
The uniform convergence in Equation~\eqref{eq:laplace}
allows us to apply this rescaling to the Laplace transform
of $U_2, \ldots, U_q$
\[
  \LL^{(\vect{U})}_{n_1,n_2}\left( \frac{\vect{t}}{\sqrt{k}} \right)
  =
  B\left( \frac{\vect{t}}{\sqrt{k}} \right) e^{k\chi(\vect{t}/\sqrt{k})}
  \left( 1 + \smallo_W(1) \right)
\]
Since $B$ is continuous and its value at $\vect{0}$ is $1$,
we can rewrite this expression as
\begin{equation} \label{eq:laplaceU}
  \LL^{(\vect{U})}_{n_1,n_2}\left( \frac{\vect{t}}{\sqrt{k}} \right)
  =
  e^{\sqrt{k} \nabla \chi(\vect{0}) \cdot \vect{t}}
  e^{\frac{1}{2}
  \vect{t} \cdot \HH \cdot \vect{t}}
  \left( 1 + \smallo_W(1) \right).
\end{equation}
It follows that, in a fixed neighborhood of $\vect{0}$, we have
\begin{equation}\label{eq:limit_laplace}
  \LL^{(\vect{V})}_{n_1,n_2}(\vect{t})
  =
  e^{\frac{1}{2}
  \vect{t} \cdot \HH \cdot \vect{t}}
  \left( 1 + \smallo_W(1) \right).
\end{equation}

From \eqref{eq:limit_laplace} above, in order to complete the proof of the convergence in distribution 
of $\vect{V}=(V_2,V_3,\dots, V_q)$, it suffices to verify that the right-hand side above is the 
Laplace transform of a multivariate  Gaussian, which reduces to checking that the Hessian matrix 
$\HH$ is positive semi-definite. 
To see that this is the case, it suffices to consider the 
cumulant generating function $\log \LL_{n_1,n_2}^{(\vect{V})}(\vect{t})$, which is a 
convex function for every $n_1$ and $n_2$ \cite[see, e.g.,][]{DeZe1998}. It follows that 
the quadratic form $\vect{t} \HH \vect{t}$ is convex, so that the matrix $\HH$ is positive 
semi-definite.

Finally, we turn to the evaluation of $\nabla \chi(\vect{0})$ and $\HH$. 
Remark that, although this is not explicit in the notation, both $\nabla \chi(\vect{0})$ and $\HH$ 
depend on $\alpha$.
By definition of the Gradient and the Hessian matrix, 
with the convention that rows and columns are indexed
from $2$ to $q$, the $j$th component of $\nabla \chi(\vect{t})$
is $\partial_{t_j} \chi(\vect{0})$
and the coefficient $(i,j)$ of $\HH$ is $\partial_{t_i} \partial_{t_j} \chi(\vect{0})$.
Since
\[
  \chi(\vect{t}) =
  \left.
  \left(
    \log \left(
      \frac{\Line(\zu, \vect{u})}{\Line(\zo, \vect{1})}
    \right)  
    - \alpha \log \left( \frac{\zu}{\zo} \right)
  \right)\right|_{\vect{u} = e^{\vect{t}}},
\]
we have
\begin{align*}
  \partial_{t_j} \chi(\vect{t})
  &=
  e^{t_j} \partial_{u_j} \left.\left( \log( \Line(\zu, \vect{u})) 
  - \alpha \log( \zu) \right)\right|_{\vect{u} = e^{\vect{t}}}
  \\&=
  e^{t_j} \left. \left(
    \partial_{u_j} \log( \Line(z, \vect{u}))_{|z = \zu}
    + ( \partial_{u_j} \zu ) \partial_z \log ( \Line(\zu, \vect{u}) )
    - \frac{\alpha}{\zu} \partial_{u_j} \zu
  \right)\right|_{\vect{u} = e^{\vect{t}}}.
\end{align*}
By the definition~\eqref{eq:zeta} of $\zeta_{\vect{u}}$,
\[
  \partial_z \log \left( \Line(\zu, \vect{u}) \right) - \frac{\alpha}{\zu} = 0,
\]
so
\begin{equation} \label{eq:gradiant}
  \partial_{t_j} \chi(\vect{t}) 
  = e^{t_j}
  \partial_{u_j} \left.\log( \Line(z, \vect{u}))\right|_{z = \zeta_{e^{\vect{t}}}, \vect{u} = e^{\vect{t}}}.
\end{equation}
According to the value of $\zeta_{\vect{1}}$ derived in Lemma~\ref{th:Gexact2}
and the expression of $\Line(z, \vect{u})$, we have
\begin{align*}
  \zo &= \frac{\alpha}{1+\alpha}, \\
  \Line(\zo, \vect{1}) &= 1+\alpha,\\
  \partial_{u_j} \Line(z,\vect{u}) &= z^{j-2}.
\end{align*}
It follows that the $j$th component of the Gradient $\nabla \chi(\vect{0})$ is
\[
  \partial_{t_j} \chi(\vect{0}) 
  = \frac{\partial_{u_j} \Line(\zo, \vect{u})|_{\vect{u} = \vect{1}}}{\Line(\zo, \vect{1})}
  = \frac{\alpha^{j-2}}{(1+\alpha)^{j-1}}.
\]
Now we compute the coefficient $(i,j)$ of the Hessian matrix $\HH$.
Let $f(z, \vect{u})$ denote the function
\[
  f(z, \vect{u}) 
  = \log (\Line(z, \vect{u}))
  = \log \Bigg( 
  \frac{1}{1-z} + \sum_{j=2}^q (u_j-1) z^{j-2}
  \Bigg),
\]
then by derivation of Equation~\eqref{eq:gradiant},
\[
  \partial_{t_i} \partial_{t_j} \chi(\vect{0})
  =
  ( \partial_{u_i} \zo ) \partial_z \partial_{u_j} f(\zo, \vect{1}) 
  + \partial_{u_i} \partial_{u_j} f(\zo, \vect{1})
  + \mathbb{1}_{i=j} \partial_{u_i} f(\zo,\vect{1}).
\]
Deriving Equation~\eqref{eq:zeta} with respect to $u_i$
and rearranging the terms leads to
\[
  \partial_{u_i} \zo
  =
  - \frac{ \zo \partial_{u_i} \partial_z f(\zo, \vect{1}) }
    {\partial_z f(\zo, \vect{1}) + \zo \partial_z^2 f(\zo, \vect{1})},
\]
so
\begin{equation} \label{eq:H}
  \partial_{t_i} \partial_{t_j} \chi(\vect{0})
  =
  - \frac{ \zo \partial_{u_i} \partial_z f(\zo, \vect{1}) }
    {\partial_z f(\zo, \vect{1}) + \zo \partial_z^2 f(\zo, \vect{1})}
  \partial_z \partial_{u_j} f(\zo, \vect{1})
  + \partial_{u_i} \partial_{u_j} f(\zo, \vect{1})
  + \mathbb{1}_{i=j} \partial_{u_i} f(\zo,\vect{1}),
\end{equation}
where $\mathbb{1}_{i=j}$ denotes the indicator that $i=j$.
Simple computations on the expression of $f(z,\vect{u})$ yield
\begin{align*}
  \partial_z f(\zo, \vect{1}) &= 
    1+\alpha, \\
  \partial_z^2 f(\zo, \vect{1}) &=
    (1+\alpha)^2, \\ 
  \partial_{u_i} f(\zo, \vect{1}) &= \frac{\alpha^{i-2}}{(1+\alpha)^{i-1}}\\
  \partial_{u_i} \partial_z f(\zo, \vect{1}) &= 
    \frac{\alpha^{i-3}}{(1+\alpha)^{i-2}} (i-2-\alpha), \\
  \partial_{u_i} \partial_{u_j} f(\zo, \vect{1}) &= 
    - \frac{\alpha^{i+j-4}}{(1+\alpha)^{i+j-2}}.
\end{align*}
Injecting those relations in Equation~\eqref{eq:H} leads to
\[
  \HH_{i,j} = 
  - \frac{\alpha^{i+j-4}}{(1+\alpha)^{i+j-2}}
  \left( 1 + \frac{(i-2-\alpha)(j-2-\alpha)}{\alpha(1+\alpha)} \right)
  + \frac{\mathbb{1}_{i=j}}{1+\alpha} \left(\frac \alpha{1+\alpha} \right)^{i-2},
\]
which completes the proof.
\end{proof}

    \section{Configuration model} \label{sec:multi}

The \emph{multigraph process}, 
also known as the \emph{uniform graph model},
produces a random vertex-labelled multigraph 
with $n$ vertices and $m$ edges
by drawing $2m$ labelled vertices
$v_1 w_1 \ldots v_m w_m$
uniformly and independently in $[1,n]$,
and adding to the multigraph the edges $\edge{v_i w_i}$
for $i$ from $1$ to $m$: 
\[
  \Edge(G) = \{ \edge{v_i w_i}\ |\ 1 \leq i \leq m\}.
\]
When the set of degrees of the output
is constrained, the multigraph process
and the configuration model generate 
the same distribution on multigraphs.
The number of \emph{sequences of vertices} $v_1 w_1 \ldots v_m w_m$
that correspond to a given multigraph $G$ is denoted by $\seqv(G)$
\[
  \seqv(G) = |
  \{ v_1 w_1 \ldots v_m w_m\ |\ 
    \{ \edge{v_i w_i}\ |\ 1 \leq i \leq m\} = \Edge(G) \} |.
\]
Observe that a multigraph~$G$ with $m$ edges
contains neither loops nor multiple edges 
(in which case it is \emph{simple})
if and only if 
its number of sequences of vertices $\seqv(G)$
is equal to $2^m m!$.
For this reason, \cite{JKLP93} introduced
the \emph{compensation factor} 
$\kappa(G) = \frac{\seqv(G)}{2^m m!}$.
The probability for the multigraph process
to produce a multigraph $G$ 
in a family $\mathcal{F}$
is then proportional to $\kappa(G)$.
Therefore, we associate to $\mathcal{F}$
the generating function
\[
  F(z) = \sum_{G \in \mathcal{F}} \kappa(G) \frac{z^{n(G)}}{n(G)!},
\]
where $n(G)$ denotes the number of vertices of~$G$.

With this convention, the generating function of paths
is the same in the multigraph process as before,
because paths are simple multigraphs
and their compensation factors are equal to one.
The generating function of cycles becomes
\[
  \Cycle(z, \vect{u}) 
  = 
  \frac{1}{2} \log \frac{1}{1-z} + \sum_{j=1}^q (u_j-1) \frac{z^j}{2 j}.
\]
because in the multigraph process,
a cycle of size one is a loop with compensation factor~$1/2$,
and a cycle of size two is a double-edge with compensation factor~$1/2$.
There is now one more variable $u_1$ that marks
the components of size~$1$, which correspond to the loops.

The set~$\setmg_{n_1, n_2}$ contains the multigraphs
with $n_1$ vertices of degree~$1$ and $n_2$ of degree~$2$.
It is equipped with the distribution induced by the compensation factors
or, equivalently, by the configuration model.
We redefine the generating function~$\MG_{n_1, n_2}(\vect{u})$
such that the sum of the compensation factors of multigraphs in~$\setmg_{n_1,n_2}$
with $m_i$ components of size~$i$ for all $1 \leq i \leq q$ is
\[
  [u_1^{m_1} \ldots u_q^{m_q}] \MG_{n_1, n_2}(\vect{u}).
\]
With those new definitions, 
Lemma~\ref{th:Gexact}, 
Corollary~\ref{th:integralform}, 
Lemma~\ref{th:Gexact2} 
and Corollary~\ref{th:Gapprox}
hold true.
Observe that~$\zeta_{\vect{u}}$ has the same value as for simple graphs,
because its implicit characterization~\eqref{eq:zeta}
only depends on the generating function~$\Line(z, \vect{u})$,
which is unchanged.

For an integer $q\ge 2$, we define the new random variables $V_2, \ldots, V_q$
as in Lemma~\ref{th:limitdistribution},
and set $V_1 = U_1$.
They are gathered into a vector $\mathbf{V} = (V_1, \ldots, V_q)$.
Following the proof of the lemma,
the multivariate Laplace transform
of the variables $\vect{V}$ is
\[
  \LL^{(\vect{V})}_{n_1,n_2}(\vect{t})
  =
  \frac{\MA(0, (e^{t_1}, 1, \ldots, 1))}{\MA(0,\vect{1})}
  e^{\frac{1}{2}
  \smat{t_2 & \cdots & t_q} \cdot \HH \cdot \smat{t_2 & \cdots & t_q}}
  \left( 1 + \smallo_W(1) \right).
\]
By definition,
$
  \MA(\theta,\vect{u}) = \exp \left( \Cycle(\zeta_{\vect{u}} e^{i \theta}, \vect{u} ) \right)
$,
so
\[
  \MA(0,\vect{u}) =
  \exp \left( \frac{1}{2} \log \left( \frac{1}{1-\zeta_{\vect{u}}} \right)
  + \sum_{j=1}^q (u_j-1) \frac{\zeta_{\vect{u}}^j}{2j}
  \right)
\]
and we know that $\zeta_{\vect{1}} = \frac{\alpha}{1+\alpha}$.
Therefore, injecting the relation
\[
  \frac{\MA(0, (e^{t_1}, 1, \ldots, 1))}{\MA(0,\vect{1})}
  =
  e^{\frac{1}{2} \frac{\alpha}{1+\alpha} (t_1 - 1)}
\]
in the expression of the Laplace transform, we obtain
\[
  \LL^{(\vect{V})}_{n_1,n_2}(\vect{t})
  =
  e^{\frac{1}{2} \frac{\alpha}{1+\alpha} (t_1 - 1)}
  e^{\frac{1}{2}
  \vect{t} \cdot \HH \cdot \vect{t}}
  \left( 1 + \smallo_W(1) \right).
\]
It follows that the limit law of $V_1$ is Poisson with parameter 
$\frac{1}{2} \frac{\alpha}{1+\alpha}$, while the limit law of $V_2, \ldots, V_q$
is Gaussian.

\begin{theorem}
Let $\alpha>0$. For $n_1$ an even let $n_2=\lfloor \alpha n_1/2\rfloor$. 
For $j \ge 1$, let $U_j$ denote the number of connected components of size $j$ in 
a multigraph produced by the configuration model in $\setg_{n_1,n_2}$. 
Then, for every $j \geq 2$, as $n_1 \to \infty$ along the 
even integers, 
\[
  \esp [U_j] \sim \frac{\alpha^{j-2}}{(1+\alpha)^{j-1}} \frac {n_1}2
  \qquad \text{and}\qquad
  \var(U_j) = O\left(\frac{n_1}2\right),
\]
and
\[
  \esp [U_1] \sim \frac{1}{2} \frac{\alpha}{1+\alpha}
  \qquad \text{and}\qquad
  \var(U_1) \sim \frac{1}{2} \frac{\alpha}{1+\alpha}.
\]
Furthermore, for any integer $q \ge 2$, as $n_1\to\infty$ along the even integers, 
the vector
\[
\frac{1}{\sqrt{n_1/2}}
  \left(
  U_j
  - \frac{\alpha^{j-2}}{(1+\alpha)^{j-1}}
  \frac{n_1}{2}
  \right)_{2\le j\le q}
\]
converges in distribution to a multivariate Gaussian $\mathcal{N}(\vect{0}, \HH(\alpha))$,
where the positive semi-definite matrix $\HH$ is defined in Theorem~\ref{thm:main},
and $U_1$ converges in distribution to a Poisson random variable of parameter 
$\frac{1}{2} \frac{\alpha}{1+\alpha}$.
\end{theorem}

\setlength{\bibsep}{.3em}
\bibliographystyle{plainnat}
\bibliography{remco}

\begin{thebibliography}{5}
\providecommand{\natexlab}[1]{#1}
\providecommand{\url}[1]{\texttt{#1}}
\expandafter\ifx\csname urlstyle\endcsname\relax
  \providecommand{\doi}[1]{doi: #1}\else
  \providecommand{\doi}{doi: \begingroup \urlstyle{rm}\Url}\fi

\bibitem[Dembo and Zeitouni(1998)]{DeZe1998}
A.~Dembo and O.~Zeitouni.
\newblock \emph{Large Deviation Techniques and Applications}.
\newblock Springer, second edition, 1998.

\bibitem[Flajolet and Sedgewick(2009)]{FS09}
P.~Flajolet and R.~Sedgewick.
\newblock \emph{Analytic Combinatorics}.
\newblock Cambridge University Press, Cambridge, UK, 2009.

\bibitem[Giardin{\`a} et~al.(2014)Giardin{\`a}, Giberti, van~der Hofstad, and
  Prioriello]{GGHP14}
C.~Giardin{\`a}, C.~Giberti, R.~van~der Hofstad, and M.L. Prioriello.
\newblock {Quenched central limit theorems for the Ising model on random
  graphs}.
\newblock \emph{Submitted}, 2014.

\bibitem[Hwang(1998)]{H98}
H.K. Hwang.
\newblock On the convergence rates in the central limit theorems for
  combinatorial structures.
\newblock \emph{European Journal of Combinatorics}, 19:\penalty0 329---343,
  1998.

\bibitem[Janson et~al.(1993)Janson, Knuth, {\L}uczak, and Pittel]{JKLP93}
S.~Janson, D.E. Knuth, T.~{\L}uczak, and B.~Pittel.
\newblock {The birth of the giant component}.
\newblock \emph{Random Structures and Algorithms}, 4\penalty0 (3):\penalty0
  233--358, 1993.

\end{thebibliography}

\end{document}